\documentclass[a4paper,11pt]{article}

\usepackage[T1]{fontenc}
\usepackage[latin1]{inputenc}
\usepackage{pgfplots}
\usepgfplotslibrary{fillbetween}

\usepackage{geometry}
\usepackage[hidelinks]{hyperref}

\usepackage{amsthm}
\usepackage{amsmath}
\usepackage{amsfonts}
\usepackage{amssymb}
\usepackage{mathdots}

\usepackage{times}
\usepackage{bm}
\usepackage{dutchcal}

\usepackage{tikz-cd}
\usepackage{fancybox}
\usepackage{caption}
\usepackage{subcaption}
\usepackage{graphicx}
\pgfplotsset{compat=1.14}

\usepackage[toc,page]{appendix}
\usepackage{todonotes}

\DeclareMathOperator{\vol}{Vol}
\DeclareMathOperator{\pr}{pr}

\DeclareMathOperator{\SU}{SU}

\DeclareMathOperator{\wt}{wt}

\DeclareMathOperator{\Tr}{Tr}

\DeclareMathOperator{\Ad}{Ad}

\DeclareMathOperator{\UU}{U}
\DeclareMathOperator{\GL}{GL}
\newcommand{\n}{^{-1}}

\usepackage{theoremref}
\newtheorem{Thm}{Theorem}[section]

\newtheorem{Pro}[Thm]{Proposition}
\newtheorem{Lem}[Thm]{Lemma}
\newtheorem{Cor}[Thm]{Corollary}

\theoremstyle{definition}

\newtheorem{Ex}[Thm]{Example}

\theoremstyle{remark}
\newtheorem{Rmk}[Thm]{Remark}

\newcommand\nnfootnote[1]{%
  \begin{NoHyper}
  \renewcommand\thefootnote{}\footnote{#1}%
  \addtocounter{footnote}{-1}%
  \end{NoHyper}
}

\usepackage{tikzsymbols}

\begin{document}
\title{Concentration of symplectic volumes on Poisson homogeneous spaces}
\author{Anton Alekseev\and Benjamin Hoffman\and Jeremy Lane \and Yanpeng Li}

\newcommand{\Addresses}{{
  \bigskip
  \footnotesize

  \textsc{Section of Mathematics, University of Geneva, 2-4 rue du Li\`evre, c.p. 64, 1211 Gen\`eve 4, Switzerland}\par\nopagebreak
  \textit{E-mail address}: \texttt{Anton.Alekseev@unige.ch}

  \medskip

  \textsc{Department of Mathematics, Cornell University, 310 Malott Hall, Ithaca, NY 14853, USA}\par\nopagebreak
  \textit{E-mail address}: \texttt{bsh68@cornell.edu}
  
  \medskip
  
  \textsc{Section of Mathematics, University of Geneva, 2-4 rue du Li\`evre, c.p. 64, 1211 Gen\`eve 4, Switzerland}\par\nopagebreak
  \textit{E-mail address}: \texttt{lane203j@gmail.com}
  
  \medskip

  \textsc{Section of Mathematics, University of Geneva, 2-4 rue du Li\`evre, c.p. 64, 1211 Gen\`eve 4, Switzerland}\par\nopagebreak
  \textit{E-mail address}: \texttt{yanpeng.li@unige.ch}

}}
\date{}
\maketitle

\nnfootnote{\emph{Keywords:} Poisson-Lie groups, homogeneous spaces, coadjoint orbits, symplectic geometry}

\begin{abstract}
For a compact Poisson-Lie group $K$, the homogeneous space $K/T$ carries a family of symplectic forms $\omega_\xi^s$, where $\xi \in \mathfrak{t}^*_+$ is in the positive Weyl chamber and $s \in \mathbb{R}$. The symplectic form $\omega_\xi^0$ is identified with the natural $K$-invariant symplectic form on the $K$ coadjoint orbit corresponding to $\xi$. The cohomology class of $\omega_\xi^s$ is independent of $s$ for a fixed value of $\xi$.

In this paper, we show that as $s\to -\infty$, the symplectic volume of $\omega_\xi^s$ concentrates in arbitrarily small neighborhoods of the smallest Schubert cell in $K/T \cong G/B$. This strengthens an earlier result of \cite{LT} and is a step towards a conjectured construction of global action-angle coordinates on $\operatorname{Lie}(K)^*$
\cite[Conjecture 1.1]{ALL}.

\end{abstract}

\section{Introduction}

Let $K$ be a compact connected Lie group with maximal torus $T$ and let $G=K^{\mathbb{C}}$ denote its complexification.  Let $\mathfrak{t}$ denote the Lie algebra of $T$. As our results concern the homogeneous space $K/T$, we may assume without loss of generality that $K$ is semisimple and simply connected.  

The homogeneous space $K/T$ carries an interesting family of symplectic structures $\omega_\xi^s$ parameterized by $s \in \mathbb{R}$ and elements of a positive Weyl chamber, $\xi \in \mathfrak{t}^*_+$. Following \cite{LW}, the Iwasawa decomposition $G = AN_-K$ defines dual Poisson-Lie groups $(K,\pi_K)$ and $(AN_-,\pi_{AN_-})$. The symplectic leaves of $\pi_{AN_-}$ are the orbits of the so-called dressing action of $K$ on $AN_-$. Let $\mathcal{D}_{\xi} \subset AN_-$ denote the dressing orbit through $\exp(\sqrt{-1}\xi)$, where $\xi \in \mathfrak{t}^*$ is identified with an element of $\mathfrak{t}$ via the Killing form. For all $s \neq 0$ and $\xi\in \mathfrak{t}_+^*$, fix the $K$-equivariant identification of $K/T$ with $\mathcal{D}_{s\xi}$ such that $eT \mapsto \exp(s\sqrt{-1}\xi)$ and define\footnote{Note that for $s<0$, $\omega_\xi^s$ is the symplectic structure on $K/T$ defined by $-s\pi_{\lambda}$, $\lambda = -s\sqrt{-1}\xi$, where $\pi_{\lambda}$ is the Poisson structure defined by Lu in \cite[Notation 5.11]{Lu00}.}
\begin{equation}\label{pi^s}
	 \pi_\xi^s := s\pi_{AN_-}\vert_{\mathcal{D}_{s\xi}}, \quad \omega_\xi^s := (\pi_\xi^s)\n.
\end{equation}
For $s=0$ and $\xi\in \mathfrak{t}_+^*$, fix the $K$-equivariant identification of $K/T$ with the coadjoint orbit $\mathcal{O}_{\xi}$ such that $eT \mapsto \xi$ and define $\omega_\xi^0$ to be the Kostant-Kirillov-Souriau symplectic form. 

The family $\omega_\xi^s$ was studied in \cite{A97,Lu00} and has several nice properties. First, the action of $K$ on $K/T$ is Poisson: the action map $K \times K/T \to K/T$ is a Poisson map with respect to $s \pi_K$ and $\pi_\xi^s$ for all $s$ and $\xi$. In other words, $(K/T,\pi_\xi^s)$ is a \emph{Poisson homogeneous space} for $(K,s\pi_K)$. Poisson homogeneous spaces for $(K,\pi_K)$ were classified in \cite{Ka1}. Second, for a fixed value of $\xi$ the forms $\omega_\xi^s$ are isotopic for all $s\in \mathbb{R}$ \cite{A97}. It follows that for fixed $\xi$ and arbitrary $s$ the forms $\omega_\xi^s$ are cohomologous. In particular, their symplectic (Liouville) volumes are the same:
\begin{equation}\label{equal_volumes}
\vol(K/T,\omega^s_\xi) = \vol(K/T,\omega_\xi^0).
\end{equation}

Let $B \subset G$ be the positive Borel subgroup (corresponding to $\mathfrak{t}_+^*$).
The flag variety $G/B$ is isomorphic to $K/T$ and admits a stratification into Schubert cells $B w B/B$, indexed by elements $w$ of the Weyl group. The smallest Schubert cell is the point $eB \in G/B$ and the biggest Schubert cell, $B w_0 B/B$, corresponding to the longest element $w_0 \in W$, is dense in $G/B$.

It follows from \cite[Proposition 5.12]{Lu00} that the rescaled family of Poisson structures $s^{-1} \pi_\xi^s$ admits, for all $\xi$, a common limit $\pi^\infty$ when $s \to -\infty$. The Poisson structure $\pi^\infty$ coincides with the image of the standard Poisson structure $\pi_K$ under the projection map $K \to K/T$ and its symplectic leaves are exactly the Schubert cells. Theorem 2.2 in \cite{LT} implies the following:

\begin{Thm}\label{wrongtheorem}
Let $\overline{U}$ be a compact subset of the big Schubert cell $B w_0 B /B$. Then, for any $\xi \in \mathfrak{t}^*_+$ and $\varepsilon>0$, there exists $s_0 \in \mathbb{R}$ such that for $s\leqslant s_0$,
\[
  \vol\left(\overline{U},\omega_{\xi}^s\right)<\varepsilon.
\]
\end{Thm}

\begin{proof}
    Fix $\xi \in \mathfrak{t}_+^*$ and identify $K/T$ with the dressing orbit $\mathcal{D}_{s\xi}$ as above, equipped with $s\pi_{AN_-}$. Let $\operatorname{pr}_A\colon G \to A$ denote projection with respect to the Iwasawa decomposition $G = AN_-K$. Identify $\mathfrak{t}\cong \mathfrak{t}^*$ via the Killing form. With these identifications,
    \[
        \Psi_s\colon K/T \to \mathfrak{t}^*, \quad kT \mapsto  \frac{1}{s\sqrt{-1}}\log \operatorname{pr}_A(k\exp(s\sqrt{-1}\xi)),
    \]
    is a moment map for the action of $T$ on $(K/T,\omega_\xi^s)$ by left multiplication, for all $s\neq 0$ \cite[Theorem 4.13]{LuRatiu}. The $T$-fixed points, their weights, and their images under the moment map do not depend on $s$. Thus the Duistermaat-Heckman measure on the moment polytope defined by $\Psi_s$ is independent of $s$.

    Fix a compact subset $\overline{U} \subset B w_0 B /B$.   By \cite[Theorem 2.2]{LT}, there exists $r > 0$ such that 
    \[
        || \log \operatorname{pr}_A(k\exp(s\sqrt{-1}\xi)) - sw_0\sqrt{-1}\xi || < r
    \]
    for all $\xi \in \mathfrak{t}_+$, $s < 0$, and  $k \in \overline{U}$. 
    The norm $||\cdot ||$ is taken with respect to the Killing form. It follows that for fixed $\xi \in \mathfrak{t}_+$ and all $s<0$, 
    \[
        ||\Psi_s(kT) - w_0\xi|| = \left|\left| \frac{1}{s\sqrt{-1}}\log \operatorname{pr}_A(k\exp(s\sqrt{-1}\xi)) - w_0\xi\right|\right| < \frac{r}{|s|}
    \]
    for all $k \in \overline{U}$. Since the Duistermaat-Heckman is independent of $s$, this implies that $\vol(\overline{U},\omega_{\xi}^s)<\varepsilon$ for all $s<0$ sufficiently large.
\end{proof}

In other words, any compact subset of the big Schubert cell is depleted of symplectic volume as $s\to -\infty$. Since total volume is constant for fixed $\xi$, this implies that the volume concentrates in a small neighborhood of the other Schubert cells. 

\begin{Ex} 
As an illustration of this phenomenon, consider the example of $K=\SU(2)$.  Identify $\mathfrak{t}^* = \mathbb{R}$ and $\xi \in \mathfrak{t}_+^* = \mathbb{R}_{> 0}$. Let $(z,\varphi)\in (-1,1)\times (0,2\pi)$ be cylindrical coordinates on the unit-sphere $S^2 \subset \mathbb{R}^3$ and fix the $K$-equivariant identification of $K/T$ with $S^2$ such that $eT$ is identified with the pole $z=1$.  The family of symplectic forms is

\[
    \omega_\xi^{s} = 
    \left\{
    \arraycolsep=1.4pt\def\arraystretch{2.2}
    \begin{array}{cll} \displaystyle\frac{\sinh(2s\xi)}{2s(\cosh(2s\xi)+z\sinh(2s\xi))} \, dz \wedge d\varphi,\quad && s \neq 0;\\ 
    \xi dz \wedge d\varphi, \quad && s = 0.
    \end{array}\right.
\]
One can derive this formula, for instance, from \cite[Example 5.4]{Lu00}. Note that $\omega_\xi^0=\xi dz \wedge d\varphi$ are the rotation-invariant area forms on $S^2$. We leave it as an exercise to the reader to show that the cohomology class of $\omega_\xi^s$ is indeed independent of $s$ and that for $s\ll 0$ the volume concentrates near the pole $z=1$, which was identified with the smallest Schubert cell, $eB$.
\end{Ex}

In general, there are many Schubert cells in $G/B$ of positive codimension and the question of how volume arranges itself on a neighborhood of those Schubert cells when $s\ll 0$ remains.  The main result of this paper is an answer to this question (and a strengthening of Theorem \ref{wrongtheorem}):

\begin{Thm}[Main Theorem] \label{maintheorem}
	Let $U$ be an open neighborhood of the smallest Schubert cell $eB$. Then for any $\xi \in \mathfrak{t}^*_+$ and $\varepsilon>0$, there exists $s_0\in \mathbb{R}$ such that for $s\leqslant s_0$, 
	\[
		\vol\left(U,\omega^s_\xi\right)>(1-\varepsilon)\vol(K/T,\omega^s_\xi).
	\]
\end{Thm}

In other words, any compact subset of $G/B$ not containing $eB$ eventually gets depleted of symplectic volume as $s \to -\infty$. 

The remainder of the paper is devoted to setting up the proof of Theorem \ref{maintheorem}, which is given below.    
Section \ref{section2} describes the dual Poisson-Lie group $(K^*,\pi_{K^*}) := (AN_-,\pi_{AN_-})$. There are two important maps defined for $s\neq 0$, 
\begin{align*}
\mathfrak{E}_s \colon & \mathfrak{k}^*\to K^* \\
\mathfrak{L}_s \colon & \mathbb{R}^{r+m} \times \mathbb{T}^m \to K^*
\end{align*}
 which are defined in Equations \eqref{STS maps} and \eqref{detrop maps}, respectively. Here $r=\dim (T)$, $2m=\dim (K/T)$, and $\mathbb{T}^{m}$ is a compact torus of dimension $m$. The map $\mathfrak{E}_s$ is a diffeomorphism. It is $K$-equivariant with respect to the coadjoint and dressing actions and has the property that $\mathfrak{E}_s(\xi) = \exp(s\sqrt{-1}\xi)$ for all $\xi \in \mathfrak{t}^*$.
 The map $\mathfrak{L}_s$ is a diffeomorphism onto its image and the image of $\mathfrak{L}_s$ is an open dense subset of $K^*$ that is independent of $s$. The intersection $\mathfrak{L}_s(\mathbb{R}^{r+m}\times \mathbb{T}^m) \cap \mathfrak{E}_s(\mathcal{O}_{\xi})$ is an open dense subset of $\mathfrak{E}_s(\mathcal{O}_{\xi})$ for all $\xi \in \mathfrak{t}_+^*$. 
Moreover, all the maps in the following diagram are Poisson:\begin{equation}\label{main diagram}
    \begin{tikzcd}
    (\mathcal{O}_\xi, \pi^s_\xi) \arrow[r,hookrightarrow] & (\mathfrak{k}^*,\pi^s = \mathfrak{E}_s^*(s\pi_{K^*})) \arrow[r,"\mathfrak{E}_s"] & (K^*,s\pi_{K^*}) & (\mathbb{R}^{r+m}\times \mathbb{T}^{m},  \mathfrak{L}_s^*(s\pi_{K^*})). \arrow[l, swap,"\mathfrak{L}_s"]
    \end{tikzcd}
\end{equation}

There is a distinguished open subset $PT(K^*)\subset \mathbb{R}^{r+m} \times \mathbb{T}^m$ called the \emph{partial tropicalization of }$K^*$, introduced in \cite{ABHL}, equipped with a constant Poisson structure $\pi_{PT}$. As $s\to -\infty$, the Poisson structure $ \mathfrak{L}_s^*(s\pi_{K^*})$ converges to $\pi_{PT}$ uniformly on certain subsets that exhaust $PT(K^*)$ (Section \ref{section PT}). Section \ref{section3} shows that the symplectic volume of the leaves of $ \mathfrak{L}_s^*(s\pi_{K^*})$ concentrates in $PT(K^*)$ as $s\to -\infty$ (Proposition \ref{prop 3.5}). Section \ref{section4} contains the proof of Proposition \ref{lastcor}, which says that, under the maps in \eqref{main diagram}, points of $PT(K^*)$ correspond to points near $\mathfrak{t}_+^*\subset \mathfrak{k}^*$ when $s\ll 0$. This allows us to translate Proposition \ref{prop 3.5} into a statement about the symplectic volume of $(K/T,\omega^s_\xi)$.

\begin{proof}[Proof of Theorem \ref{maintheorem}]

Let $\mathcal{N}_{s\xi}\subset\mathbb{R}^{r+m}\times \mathbb{T}^m$ denote the preimage $( \mathfrak{E}_s\n\circ \mathfrak{L}_s )\n(\mathcal{O}_\xi)$, which is a symplectic leaf of $ \mathfrak{L}_s^*(s\pi_{K^*})$, and denote its symplectic form $\eta_{s\xi} = (\mathfrak{E}_s\n \circ \mathfrak{L}_s)^* \omega_\xi^s$. In Proposition \ref{prop 3.5}, we prove that for all $\varepsilon>0$, there is a compact subset $D_\varepsilon \subset  PT(K^*)$ such that
\[
\lim_{s\to -\infty} \vol\left(\mathcal{N}_{s\xi} \cap D_\varepsilon,\eta_{s\xi}\right) \geqslant (1-\varepsilon)\vol(\mathcal{N}_{s\xi},\eta_{s\xi}) = (1-\varepsilon)\vol(K/T,\omega_\xi^s).
\]
In Proposition \ref{lastcor}, we show there exists $s_0< 0$ such that for all $s\leqslant s_0$,  
\[
    \mathfrak{E}_s\n\circ \mathfrak{L}_s(\mathcal{N}_{s\xi} \cap D_\varepsilon) \subseteq U.
\]
Since $\mathfrak{E}_s\n\circ \mathfrak{L}_s$ is a Poisson isomorphism, it preserves volumes of the symplectic leaves. Thus 
\[
    \vol\left(U,\omega^s_{\xi}\right) \geqslant \vol\left(\mathfrak{E}_s\n\circ \mathfrak{L}_s(\mathcal{N}_{s\xi} \cap D_\varepsilon),\omega^s_{\xi}\right)= \vol\left(\mathcal{N}_{s\xi} \cap D_\varepsilon,\eta_{s\xi}\right).
\]
Combining with the limit above completes the proof.
\end{proof}

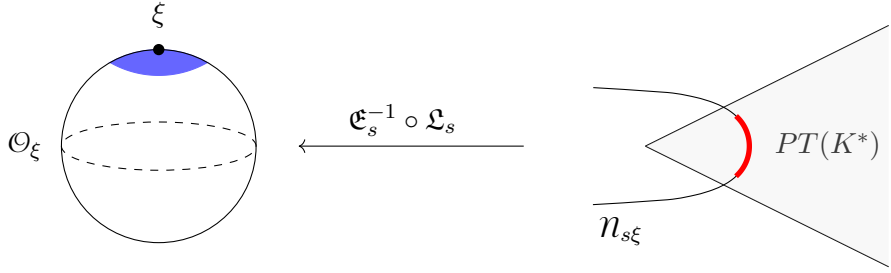
\begin{figure}[htp]
\centering
\begin{tikzpicture}[scale=1.6]	
		
		\draw[draw=blue!60!white,fill=blue!60!white] plot[smooth,samples=100,domain=-.4:.4] (\x,{sqrt(.64-\x*\x)}) -- plot[smooth,samples=100,domain=-.4:.4] (\x,{- sqrt(.64-\x*\x)+2*.693});

		\draw (0,0) circle (.8);
		\draw[dashed](0,0) ellipse (.8 and .2);
		\node at (0,.8)[circle,fill,inner sep=1.5pt]{};
		\draw (0,.9)  node[align=left, above] {$\xi$};
		\draw (-1.1,0)  node[align=left] {$\mathcal{O}_{\xi}$};

		\draw[<-] (1.15,0) -- (3,0);
		\draw (2,0)  node[align=left,above] {$\mathfrak{E}_s\n\circ \mathfrak{L}_s$};
		
		\draw (4,0)--(6,1);
		\draw (4,0)--(6,-1);
		\draw (5.5,0)  node[align=right] {$PT(K^*)$}; 
		\fill[gray!20,nearly transparent] (4,0) -- (6,1) -- (6,-1) -- cycle;
        
        \draw[domain=-.97:.97,smooth,variable=\t]plot ({4+ln(2*cosh(1*2) - 2*cosh(1*2*\t))/(2*1)},\t/2);
        
        \draw (3.8,-.5)  node[align=left,   below] {$\mathcal{N}_{s\xi}$};
        
        \draw[domain=-.5:.5,smooth,variable=\t,line width=.7mm, red]plot ({4+ln(2*cosh(1*2) - 2*cosh(1*2*\t))/(2*1)},\t/2);
        
	\end{tikzpicture}	
	\caption{As $s \to -\infty$, volume of the symplectic leaves $\mathcal{N}_{s\xi} = (\mathfrak{E}_s\n \circ \mathfrak{L}_s)\n(\mathcal{O}_\xi)$ concentrates on subsets of $\mathcal{N}_{s\xi}\cap PT(K^*)$, illustrated in red. For $s$ sufficiently large, the image of the red subset is contained in an arbitrarily small neighborhood of $\xi$, illustrated in blue. }
\label{fig 0}
\end{figure}

A motivation for our study is provided by the following idea. There exist Poisson isomorphisms between $\mathfrak{k}^*$ and $K^*$ called Ginzburg-Weinstein isomorphisms after the authors of \cite{GW}. Given a Ginzburg-Weinstein isomorphism $\gamma: \mathfrak{k}^* \to K^*$, its scaling $\gamma^s(x):=\gamma(sx)$ is a Poisson isomorphism with respect to $\pi_{\mathfrak{k}^*}$ and $s\pi_{K^*}$. 
Composing $\gamma^s$ with $\mathfrak{L}_s\n$ defines coordinates on every regular coadjoint orbit which are almost global action-angle coordinates for $s\ll0$. Conjecturally, the $s \to -\infty$ limit of this composition defines global action-angle coordinates on the regular coadjoint orbits. This has already been shown to be true for $K=\UU(n)$, where for a certain choice of Ginzburg-Weinstein diffeomorphism and cluster seed, the limit is the classical Gelfand-Zeitlin system \cite{ALL}. 

{\bf Acknowledgements.} We are grateful to J.H. Lu and R. Sjamaar for their useful comments and discussions. Research of A.A., J.L. and  Y.L. was supported in part by the grant MODFLAT of the European Research Council (ERC), by the grants number 178794 and 159581 of the Swiss National Science Foundation (SNSF) and by the NCCR SwissMAP of the SNSF. B.H. was supported by the National Science Foundation Graduate Research Fellowship under Grant Number DGE-1650441. A.A., B.H., and Y.L. express their gratitude for the hospitality and support of the Simons Center for Geometry and Physics, during their visit in 2018.

\section{Background}
\label{section2}

Fix the following notation. Let $G$ be a connected simply-connected semisimple complex Lie group of rank $r$. Fix a compact real form $K\subset G$ and a Cartan subgroup $H \subset G$, and let $(\cdot)^*:G\to G$ be the anti-involution of $G$ under which elements $k\in K$ satisfy $k\n= k^*$. Denote the  Lie algebras of $G$, $K$, and $H$ by $\mathfrak{g}$, $\mathfrak{k}$, and $\mathfrak{h}$ respectively.  Fix a choice of positive roots of $\mathfrak{g}$ with respect to $\mathfrak{h}$. Denote the lattice of integral weights by $P$, and the semigroup of dominant integral weights by $P_+$. We write $h\mapsto h^\mu\in \mathbb{C}^\times$ for the multiplicative character $H\to \mathbb{C}^\times$ determined by $\mu\in P$.   Let $I = \{1,\dots,r\}$ index the simple roots, $\alpha_i\in \mathfrak{h}^*$, the simple coroots, $\alpha_i^\vee\in \mathfrak{h}$, and the fundamental weights, $\omega_i$, which by definition satisfy $\omega_i(\alpha_j^\vee)=\delta_{ij}$.  Denote the Weyl group of $G$ by $W$. Let $s_i\in W$ be the simple reflection generated by $\alpha_i$ and let $w_0$ be the longest element of $W$, with length denoted by $m$. 

Let $T$ be the maximal torus of $K$ which has Lie algebra $\mathfrak{t}=\mathfrak{h}\cap\mathfrak{k}$. Let $\mathfrak{a} = \sqrt{-1} \mathfrak{t}$ and denote the corresponding subgroup of $G$ by $A$. Corresponding to the choice of positive roots, we have opposite maximal unipotent subgroups $N_{\pm}$ with Lie algebras $\mathfrak{n}_{\pm}$ as well as opposite Borel subgroups $B_{\pm} = HN_{\pm}$ with Lie algebras $\mathfrak{h}\oplus \mathfrak{n}_{\pm}$. Fix a set of Chevalley generators $F_i\in \mathfrak{n}_-,$ $\alpha_i^\vee\in \mathfrak{h}$, $E_i\in \mathfrak{n}$, $i \in I$. Recall the Iwasawa decompositions $G = AN_-K$ and $\mathfrak{g} = \mathfrak{n}_- \oplus \mathfrak{a} \oplus \mathfrak{k}$.

Fix an invariant non-degenerate bilinear form $(\cdot,\cdot)$ on $\mathfrak{g}$.  The isomorphism $\mathfrak{k} \cong \mathfrak{k}^*$ determined by $(\cdot,\cdot)$ embeds $\mathfrak{t}^* \subseteq \mathfrak{k}^*$, as the image of $\mathfrak{t}$. Let $\mathfrak{t}_+^* \subseteq \mathfrak{t}^*$ be the open cone such that  $\sqrt{-1}\mathfrak{t}_+^* \subseteq \mathfrak{h}^*$ is the interior of the real cone spanned by $P_+$. We refer to both $\mathfrak{t}_+^*$ and $\sqrt{-1} \mathfrak{t}_+^*$ as the positive Weyl chamber.

\subsection{Dressing orbits and compact Poisson-Lie groups}\label{section 2.1}

Recall that a Poisson-Lie group $(K,\pi)$ is a Lie group $K$ equipped with a Poisson structure $\pi$ such that the multiplication map $K\times K\to K$ is Poisson (with respect to the product Poisson structure on $K\times K$). For example, the canonical Lie-Poisson structure $\pi_{\mathfrak{k}^*}$ on the dual $\mathfrak{k}^*$ of a Lie algebra $\mathfrak{k}$ is linear, so $(\mathfrak{k}^*,\pi_{\mathfrak{k}^*})$ is a Poisson-Lie group with respect to vector addition.

For $G$ as above, both $K$ and $AN_-$ have natural Poisson-Lie group structures defined as follows (see \cite{LW} for details). Let $\Im(\cdot,\cdot)$ be the imaginary part of the fixed $G$-invariant non-degenerate bilinear form $(\cdot,\cdot)$. Then $\mathfrak{k}$ and $\mathfrak{n}_- \oplus \mathfrak{a}$ are isotropic subspaces with respect to $2\Im(\cdot,\cdot)$, and $2\Im(\cdot,\cdot)$ defines an isomorphism $\mathfrak{n}_-\oplus \mathfrak{a} \cong \mathfrak{k}^*$. This identification endows $\mathfrak{k}$ and $\mathfrak{k}^*$ with the structure of dual Lie bialgebras. Since $K$ and $AN_-$ are simply connected, the Lie bialgebra structures on $\mathfrak{k}$ and $\mathfrak{k}^*$ integrate to define Poisson-Lie group structures $\pi_K$ on $K$ and $\pi_{K^*}$ on $AN_-$, respectively. These Poisson-Lie group structures are dual, since they arise by integrating dual Lie bialgebras, thus one denotes $K^* = AN_-$, and refers to $(K^*,\pi_{K^*})$ as the \emph{dual Poisson-Lie group} of $(K,\pi_K)$.

Both $\mathfrak{k}^*$ and $K^*$ have naturally defined $K$ actions. The \emph{coadjoint action} of $K$ on $\mathfrak{k}^*$ is defined in terms of the adjoint action by the equation
\[
  \langle \Ad_k^*\xi,x\rangle=\langle \xi, \Ad_{k\n} x\rangle,\qquad k\in K,\,\xi \in \mathfrak{k}^*, \text{ and } x\in \mathfrak{k}.
\]
The coadjoint action preserves $\pi_{\mathfrak{k}^*}$, and the symplectic leaves of $\pi_{\mathfrak{k}^*}$ are the coadjoint orbits. The \emph{dressing action} of $K$ action on $K^*$ is defined by re-factorizing $kb\in G$ according to the Iwasawa decomposition. If
\[
  kb=b'k'\in AN_-K, \qquad k,k'\in K,~b,b'\in K^*,
\]
then the dressing action of $k$ on $b$ is defined as $^k b=b'$. The symplectic leaves of $\pi_{K^*}$ are the dressing orbits. In other words, they are the joint level sets of the Casimir functions \cite{LW},
\begin{equation}\label{Casimir functions}
  C_i(b)^2:=\Tr\left(\rho^{\omega_i}\left(bb^*\right)\right), \qquad b\in K^*,
\end{equation}
where $\rho^{\omega_i}$ is the fundamental irreducible $G$-representation with highest weight $\omega_i\in P_+$. The map $\varphi \colon b \mapsto bb^*$ is a diffeomorphism of $K^*$ onto the set $S=\{g\in G\mid g^*= g\}$.

There is a family of diffeomorphisms $\mathfrak{E}_s\colon \mathfrak{k}^* \to K^*$ parameterized by $s\neq 0$ \cite{FR}. Let $\psi\colon \mathfrak{k}^*\to \mathfrak{k}$ be the $K$-equivariant isomorphism given by the fixed bilinear form on $\mathfrak{g}$. Then, define
\begin{equation}\label{STS maps}
	\mathfrak{E}_s\colon \mathfrak{k}^*\xrightarrow{~~\psi~~} \mathfrak{k} \xrightarrow{\exp\left(2s\sqrt{-1}\cdot\right)} S \xrightarrow{\varphi^{-1}} K^*= AN_-.
\end{equation}

The map $\mathfrak{E}_s$ is equivariant with respect to the coadjoint and dressing actions of $K$. Let $\mathcal{O}_{\xi}$ be the coadjoint orbit through $\xi\in  \mathfrak{t}_+^*$. Denote by $\mathcal{D}_{s\xi}$ the dressing orbit through $\mathfrak{E}_s(\xi) = \exp\left(  s\sqrt{-1}\psi(\xi)\right)$. Since $\mathfrak{E}_s$ is $K$-equivariant, $\mathfrak{E}_s(\mathcal{O}_\xi) = \mathcal{D}_{s\xi}$. 

\subsection{Cluster coordinates on double Bruhat cells}\label{section 2.2}

The \emph{double Bruhat cell} determined by a pair of elements $u,v\in W$, is the intersection
\[
  G^{u,v}:=BuB\cap B_-vB_- \subset G.
\]
In particular, we will consider $G^{w_0,e}=Bw_0B\cap B_-$, which is an open dense subset of $B_-$.

Let $G_0=N_-HN$ be the open dense subset of elements in $G$ that admit a Gaussian decomposition.  
For a dominant weight $\mu\in P_+$, the \emph{principal minor} $\Delta_{\mu,\mu}$ is a regular function $G\to \mathbb{C}$ uniquely determined by its value on $G_0$:
\[
  \Delta_{\mu,\mu}(n_-hn)=h^\mu, \mbox{~for any~} n_-\in N_-, h\in H, n\in N.
\]
For any two weights $\gamma$ and $\delta$ of the form $\gamma=w\mu$, $\delta=v\mu$, for some $w,v\in W$, the {\em generalized minor} $\Delta_{w\mu, v\mu}$ is the regular function on $G$ given by
\[
  \Delta_{\gamma,\delta}(g)=\Delta_{ w\mu, v\mu}(g)=\Delta_{\mu,\mu}(\overline  w^{\,-1}g\overline{v}), \mbox{~for~} g\in G,
\]
where $\overline{w}$ is a specific lift of $w\in W$ to $G$ as in \cite[Equation~1.5]{BKII}.

Fix a reduced word $\mathbf{i}=(i_1,\dots,i_m)$, $i_j\in I$, for $w_0=s_{i_1}\cdots s_{i_m}$. Let $\bm{R}= \bm{R^-}\cup \bm{R^+}$, where $\bm{R^-}=[-r,-1]$ and $\bm{R^+}=[1,m]$. 
For $1<k<m$, let $v_k=s_{i_m}\cdots s_{i_{k+1}}$ and let $v_m=e$. For $k\in \bm{R}^-$, let $i_k=-k$ and $v_k=w_0$. Consider the functions
\[
    \Delta_{k}:=\Delta_{v_k\omega_{i_k},\omega_{i_k},} \quad k\in \bm{R}.
\] The functions $\Delta_k$ form a seed for the upper cluster algebra structure on $\mathbb{C}[G^{w_0,e}]$ described in \cite{BFZ}. 

Being an upper cluster algebra implies that any $f\in \mathbb{C}[G^{w_0,e}]$ is a Laurent polynomial in the functions $\Delta_k$. The functions $\Delta_k$ then determine an open embedding
\begin{equation}\label{clusterfor1connected}
  \sigma(\mathbf{i}) \colon(\mathbb{C}^\times)^{m+r}\to G^{w_0,e},
\end{equation}
which is a (birational) inverse to
\[
G^{w_0,e}\to \mathbb{C}^{m+r};\qquad g\mapsto (\Delta_{-r}(g),\dots, \Delta_m(g)).
\]
Note that there is no term $\Delta_k$ with index $k=0$. 

We conclude this section by recalling how generalized minors appear in matrix entries of representations of $G$. A dominant integral weight $\mu \in P_+$ can be written uniquely as
\[
\mu = \sum_{i\in I} c_i(\mu) \omega_i, \qquad c_i(\mu)\in \mathbb{Z}_{\geqslant 0}.
\]
Then the function $\Delta_{w_0\mu,\mu}$  can be written as
\begin{equation}
\label{deltadef}
\Delta_{w_0\mu,\mu} = \prod_{i\in I} \Delta_{w_0\omega_i,\omega_i}^{c_i(\mu)}.
\end{equation}
One can check that
\[
  h\cdot \Delta_{w_0\mu,\mu} \cdot h'= h^{-w_0\mu}{h'}^{\mu} \Delta_{w_0\mu,\mu},  \qquad E_i\cdot \Delta_{w_0\mu,\mu}=\Delta_{w_0\mu,\mu}\cdot E_i=0 \text{~for~} i\in I,
\]
where $h,h'\in H$, and $G$ acts on $\mathbb{C}[G]$ in the standard way
\[
  (g\cdot f\cdot h)(x)=f(g\n x h)\qquad g,h,x\in G,~f\in \mathbb{C}[G].
\]

For a sequence of indices $\mathbf{j}=(j_1,\dots,j_n)$ in $I$, write $F_{\mathbf{j}}=F_{j_1}F_{j_2}\cdots F_{j_n} \in U(\mathfrak{g})$. Recall that the functions $F_{\mathbf{j}}\cdot \Delta_{w_0\mu,\mu} \cdot F_{\mathbf{k}}$ arise from representations of $G$ as follows. Let $(V,\rho\colon G\to \GL(V))$ be the irreducible $G$-module with highest weight $\mu$. Let $v_1,\dots, v_n$ be a weight basis of $V$, where $H$ acts on $v_j$ with weight $\wt(v_j)\in \mathfrak{h}^*$, and assume $\wt(v_1)=\mu$ and $\wt(v_n)=w_0 \mu$. Let $\rho_{j,k}(g)$ be the $(j,k)$-entry of the matrix for $\rho(g)$ with respect to the basis $\{v_j\}$. Then $\rho_{n,1}=c \Delta_{w_0 \mu,\mu}$, for some $c\in \mathbb{C}^\times$. We may choose the weight basis such that $c=1$. Each $\rho_{j,k}$ is a linear combination of terms of the form $F_{\mathbf{j}} \cdot \Delta_{w_0\mu,\mu} \cdot F_{\mathbf{k}}$, where $\mathbf{j}$ and $\mathbf{k}$ are such that
\begin{equation}
\label{weightcondition}
h\cdot (F_{\mathbf{j}} \cdot \Delta_{w_0\mu,\mu} \cdot F_{\mathbf{k}}) \cdot h' = h^{-\wt(v_j)} (h')^{\wt(v_k)} (F_{\mathbf{j}} \cdot \Delta_{w_0 \mu,\mu} \cdot F_{\mathbf{k}})
\end{equation}
for all $h,h'\in H$.

\subsection{The partial tropicalization and its symplectic leaves}\label{section PT}

Recall from Section \ref{section 2.1} that $K^*=AN_-$. Let $\bm{S}=\{k\in \bm{R} \mid v_k\omega_{i_k}\neq\omega_{i_k} \}$. Then $|\bm{R}\backslash\bm{S}|=r$, and $\Delta_k(K^*)\subset \mathbb{R}_+$ if and only if $k\in \bm{R}\backslash \bm{S}$. The collection of functions
\[
    \{\Delta_k \mid k\in \bm{R}\}\cup \{\overline{\Delta_k} \mid k \in \bm{S}\}
\]
define a real coordinate system on an open dense subset of $K^*$. Equip $\mathbb{R}^{r+m}\times \mathbb{T}^m$ with coordinates $(\lambda_{\bm{R}},\varphi_{\bm{S}})$, where $\lambda_{\bm{R}}=(\lambda_k)_{k \in \bm{R}}$ and 
$\varphi_{\bm{S}}=(\varphi_k)_{k \in \bm{S}}$.

There is a Poisson manifold $(PT(K^*),\pi_{PT})$, called the \emph{partial tropicalization of} $K^*$, which was introduced in \cite{ABHL}. As a manifold, $PT(K^*)$ is defined as
\[
  PT(K^*) := \mathcal{C}\times \mathbb{T}^m \subset \mathbb{R}^{r+m} \times \mathbb{T}^m,
\]
where $\mathcal{C}$ is an open convex polyhedral cone of dimension $r+m$ defined by inequalities described in \cite{BKII} and \cite[Theorem 6.24]{ABHL}. The definition of $\mathcal{C}$ depends on the choice of reduced word $\mathbf{i}$ fixed in Section \ref{section 2.2}. More precisely, $\mathcal{C}$ is the set of points $x\in \mathbb{R}^{m+r}$ satisfying an inequality $\Phi^t(x)>0$, where $\Phi^t: \mathbb{R}^{m+r}\to \mathbb{R}$ is a certain piecewise-linear function called the tropical Berenstein-Kazhdan potential.
The Poisson structure $\pi_{PT}$ is constant in the coordinates $(\lambda_{\bm{R}},\varphi_{\bm{S}})$. The symplectic leaves of $PT(K^*)$ are the joint level sets of the coordinates $\lambda_{\bm{R}^-} = (\lambda_{-r},\dots,\lambda_{-1})$ \cite[Theorem~6.5]{ABHL2}. 

There is a correspondence between symplectic leaves of $PT(K^*)$ and regular coadjoint orbits of $K$, which we now describe. To each $\xi\in \mathfrak{t}^*_+$ we associate $\lambda_{\bm{R}^-}\in \mathbb{R}^r$ with coordinates
\[
\lambda_{-i} = ( w_0\omega_i, \sqrt{-1}\xi) \mbox{ for }i = -r, \ldots ,-1.
\]
Denote the symplectic leaf of $PT(K^*)$ that is the fiber of $\lambda_{\bm{R}^-}$ by $\mathcal{P}_{\xi}$. The corresponding coadjoint orbit is $\mathcal{O}_{\xi}$. For each fixed value of $s\neq 0$, the leaf $\mathcal{P}_{\xi}$ also corresponds to the dressing orbit $\mathcal{D}_{s\xi}$, defined in Section \ref{section 2.1}, 

Each symplectic leaf $\mathcal{P}_{\xi} \subset PT(K^*)$ inherits a symplectic form from $\pi_{PT}$  denoted by $\nu_\xi$. 

\begin{Thm}\cite[Theorem 6.11]{ABHL2}
\label{theoremcomparevolume}
The symplectic volume of  $(\mathcal{P}_{\xi},\nu_{\xi})$  equals the symplectic volume of the coadjoint orbit $\mathcal{O}_{\xi} \subset \mathfrak{k}^*$ equipped with the Kirillov-Kostant-Souriau symplectic form:
  \[
  \vol\left(\mathcal{P}_{\xi} ,\nu_\xi\right)= \vol (\mathcal{O}_{ \xi},\omega_{\xi}).
  \]
\end{Thm}
\begin{Rmk}
Although \cite[Theorem 6.11]{ABHL2} is only stated for leaves parameterized by regular dominant integral weights, the theorem here follows by scaling and continuity.
\end{Rmk}

In order to compare the Poisson structures of $PT(K^*)$ and $K^*$, we define the \emph{detropicalization map} $\mathfrak{L}_s\colon PT(K^*) \to K^*$ as follows. For $s<0$, let
\begin{equation}
\label{detrop maps}
 \mathfrak{L}_s(\lambda_{\bm{R}},\varphi_{\bm{S}})= \sigma(\mathbf{i})\left( e^{s\lambda_{-r}-\sqrt{-1}\varphi_{-r}},\dots, e^{s\lambda_{m}-\sqrt{-1}\varphi_{m}}\right),
\end{equation}
where we understand $\varphi_k=0$ on the right hand side if $k\notin \bm{S}$. Denote $b_s = \mathfrak{L}_s(\lambda_{\bm{R}},\varphi_{\bm{S}})$.

\begin{Rmk}[Conventions]
We follow the conventions of \cite{ABHL2, BKII} for (partial) tropicalization, which are opposite to those of \cite{ABHL}. We consider $K^*\subset B_-$, as in \cite{ABHL2}, rather than $K^*\subset B$, as in \cite{ABHL}, and take the limit $s\to -\infty$.  This accounts for the minus signs in \eqref{detrop maps}.
\end{Rmk}

The Casimir functions for $K^*$ are related to the coordinates $\lambda_{\bm{R}},\varphi_{\bm{S}}$ by the detropicalization map via $r$ equations (one for each Casimir function):
\begin{equation}\label{master equations}
    \begin{split}
            C_i(b_s)^2 & = \Tr(\rho^{\omega_i}(b_sb_s^*))  = \sum_{j} \rho^{\omega_i}_{j,j}(b_sb_s^*)  = \sum_{j,k}|\rho^{\omega_i}_{j,k}(b_s)|^2 \\
                     & = \sum_{j,k}\bigg\vert\sum_{\bf{i},\bf{j}} c_{\bf{i},\bf{j}}(F_{\bf{i}}\Delta_{w_0\omega_i,\omega_i} F_{\bf{j}})(b_s)\bigg\vert^2 \\
                     & = |\Delta_{w_0\omega_i,\omega_i}(b_s)|^2\left(1 +  \sum_{j,k}\bigg\vert\sum_{\bf{i},\bf{j}}c_{\bf{i},\bf{j}}\frac{ (F_{\bf{i}}\Delta_{w_0\omega_i,\omega_i} F_{\bf{j}})(b_s)}{\Delta_{w_0\omega_i,\omega_i}(b_s)}\bigg\vert^2 \right).\\
    \end{split}
\end{equation}

Since $b_s = \mathfrak{L}_s(\lambda_{\bm{R}},\varphi_{\bm{S}})$, the last line on the right side can be rewritten as a Laurent polynomial in the functions $e^{s\lambda_{k}-\sqrt{-1}\varphi_{k}}$.
The term $|\Delta_{w_0\omega_i,\omega_i}(b_s)|^2 = e^{2s\lambda_{-i}}$ dominates the expression for $s\ll0$, and the exponents in the other terms are controlled by their distance from the boundary of $\mathcal{C}$, as follows.

Recall that $\mathcal{C}$ is the set of points $x\in \mathbb{R}^{m+r}$ satisfying the inequality $\Phi^t(x)>0$. For $\delta>0$, let $\mathcal{C}^\delta\subset \mathcal{C}$ be the set of points $x\in \mathbb{R}^{m+r}$ which satisfy the inequality $\Phi^t> \delta$. Then,

\begin{Pro}\label{prop 2.4}\cite[Theorem 4.13 and Lemma 6.17]{ABHL} For $(\lambda_{\bm{R}},\varphi_{\bm{S}}) \in \mathcal{C}^{\delta}\times\mathbb{T}^m$, each term
\begin{equation*}
    \bigg\vert\sum_{\bf{i},\bf{j}}c_{\bf{i},\bf{j}}\frac{ (F_{\bf{i}}\Delta_{w_0\omega_i,\omega_i} F_{\bf{j}})(b_s)}{\Delta_{w_0\omega_i,\omega_i}(b_s)}\bigg\vert = O(e^{s\delta}).
\end{equation*}

\end{Pro}

Here and throughout, a function $f(s)$ is in $O(g(s))$, $g(s) \geqslant 0$, if there exists $c >0$ such that 
\[
    -cg(s) \leqslant f(s) \leqslant cg(s).
\]

As a direct consequence of Proposition \ref{prop 2.4} and Equations \eqref{master equations}, we have:

\begin{Cor}\cite[Remark 6.6]{ABHL2}\label{limit corollary}
    For all $\xi\in \mathfrak{t}_+^*$ and $(\lambda_{\bm{R}},\varphi_{\bm{S}}) \in \mathcal{P}_{\xi}$, and for each $i = 1,\ldots, r$,
    \[
        \lim_{s\to -\infty} \frac{1}{s}\log\circ C_i \circ \mathfrak{L}_s(\lambda_{\bm{R}},\varphi_{\bm{S}}) = \lambda_{-i}=(w_0\omega_i,\sqrt{-1}\xi).
    \]
\end{Cor}

\begin{Rmk} 
 Corollary \ref{limit corollary} says that points $\mathfrak{L}_s(\mathcal{P}_{\xi})$ in the image of a tropical leaf under the detropicalization map approach the corresponding scaled dressing orbit $\mathcal{D}_{s\xi}$ in the limit $s\to -\infty$. It is useful to note that points in $\mathfrak{L}_s(\mathcal{P}_{\xi})$ will concentrate near a certain region of $\mathcal{D}_{s\xi}$, not the entire orbit: there are points in the preimages of the scaled dressing orbits $\mathfrak{L}_s^{-1}(\mathcal{D}_{s\xi})$ that remain far away from $PT(K^*)$, even as $s\to -\infty$ (see Figure \ref{fig 1}).
\end{Rmk}

\section{Symplectic volumes of the leaves of \texorpdfstring{$\pi_s$}{pis} }\label{section3}

In this section we study volumes of the symplectic leaves of the Poisson bivector 
\[
    \pi_s := (\mathfrak{L}_s)^*(s\pi_{K^*}).
\]
Note that the pullback of a bivector under a diffeomorphism is by definition the pushforward under the inverse diffeomorphism. The symplectic leaves in question are submanifolds of  $\mathbb{R}^{r+m}\times \mathbb{T}^m$. Roughly, for $s\ll 0$ each of these leaves has a piece which lies inside $PT(K^*) = \mathcal{C}\times \mathbb{T}^m$, close to the corresponding leaf of $\pi_{PT}$ (Section \ref{ss 3.1}). For $s\ll 0$, the volume of the symplectic leaves concentrate there (Proposition \ref{prop 3.5}). This is illustrated in Figure \ref{fig 1}.

Let us first establish some notation. Each symplectic leaf of $\pi_s$ is the preimage under $\mathfrak{L}_s$ of a dressing orbit. We denote the leaf and its symplectic form by
\[
    \mathcal{N}_{s\xi} := \mathfrak{L}_s^{-1}(\mathcal{D}_{s\xi}), \qquad \eta_{s\xi} := (\pi_s)^{-1}.
\]
There is a corresponding symplectic leaf $\mathcal{P}_{\xi}$ of $PT(K^*)$ equipped with $\nu_{\xi}$, as described in Section \ref{section PT}. Recall, for $\xi\in \mathfrak{t}_+^*$,
\[
    \mathcal{P}_{\xi} := \left\{(\lambda_{\bf{R}},\varphi_{\bf{S}}) \in PT(K^*) \mid  \lambda_{-i} = ( w_0\omega_i,\sqrt{-1}\xi), \,  i = -r,\ldots, -1 \right\},
\]
which is a product of an open polytope (a fiber in $\mathcal{C}$ of projection to the first $r$ coordinates) times a torus. We will often reference the open subset $\mathcal{P}_{\xi}^{\delta} :=\mathcal{P}_{\xi}\cap(\mathcal{C}^{\delta}\times \mathbb{T}^m)$ and its closure $\overline{\mathcal{P}}_{\xi}^{\delta}$.

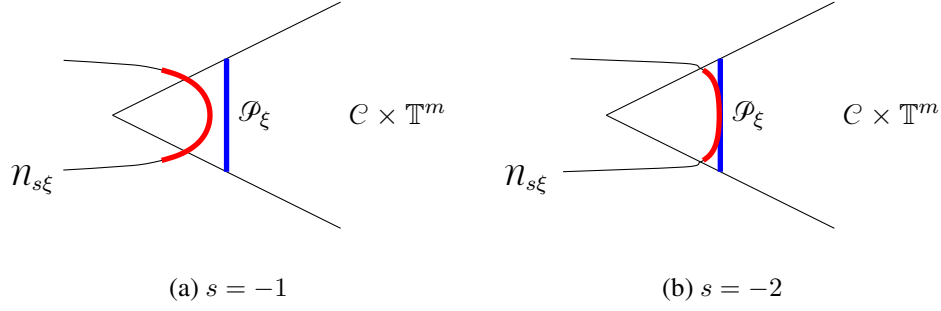
\begin{figure}[htp]
\centering
\begin{minipage}[b]{0.4\columnwidth}
\centering
\[
    \begin{tikzpicture}[baseline={([yshift=-.5ex]current bounding box.center)}, scale=1.5]
		\draw (0,0)--(2,1);
		\draw (0,0)--(2,-1);
		\draw [line width=.7mm, blue](1,-.5)--(1,.5);
        
        \draw[domain=-.97:.97,smooth,variable=\t]plot ({ln(2*cosh(1*2) - 2*cosh(1*2*\t))/(2*1)},\t/2);
        
        \draw[domain=-.8:.8,smooth,variable=\t,line width=.7mm, red]plot ({ln(2*cosh(1*2) - 2*cosh(1*2*\t))/(2*1)},\t/2);
        
        
        \draw (-.7,-.35)  node[align=left,   below] {$\mathcal{N}_{s\xi}$};
        
        \draw (1.25,0)  node[align=right] {$\mathcal{P}_{\xi}$};
        
        \draw (2.5,0)  node[align=right] {$\mathcal{C}\times \mathbb{T}^m$};
	\end{tikzpicture}
\]
\subcaption{$s = -1$}
\end{minipage}
\begin{minipage}[b]{0.4\columnwidth}
\centering
\[
\begin{tikzpicture}[baseline={([yshift=-.5ex]current bounding box.center)}, scale=1.5]
		\draw (0,0)--(2,1);
		\draw (0,0)--(2,-1);
		\draw [line width=.7mm, blue](1,-.5)--(1,.5);
        
        \draw[domain=-.999:.999,smooth,variable=\t]plot ({ln(2*cosh(2*2) - 2*cosh(2*2*\t))/(2*2)},\t/2);
        
        \draw[domain=-.8:.8,smooth,variable=\t,line width=.7mm, red]plot ({ln(2*cosh(2*2) - 2*cosh(2*2*\t))/(2*2)},\t/2);
        
        
        \draw (-.7,-.35)  node[align=left,   below] {$\mathcal{N}_{s\xi}$};
        
        \draw (1.25,0)  node[align=right] {$\mathcal{P}_{\xi}$};
        
        \draw (2.5,0)  node[align=right] {$\mathcal{C}\times \mathbb{T}^m$}; 
	\end{tikzpicture}
\]
\subcaption{$s = -2$}
\end{minipage}	
\caption{Volume of the symplectic leaves $\mathcal{N}_{s\xi}$ of $\pi_s$ concentrates on the part of $\mathcal{N}_{s\xi}$ that is close to the corresponding tropical leaf, $\mathcal{P}_{\xi}$.}
\label{fig 1}
\end{figure}

\subsection{The implicit function theorem argument}\label{ss 3.1}

Consider the map 
\begin{equation}\label{map}
    F_{s\xi} = (f_{-r},\ldots,f_{-1})\colon \mathbb{R}^r\times \mathbb{R}^m \times \mathbb{T}^m \to \mathbb{R}^r
\end{equation}
with coordinates $f_{-i}$ defined by composing the detropicalization map \eqref{detrop maps} with the Casimir functions \eqref{Casimir functions} on $K^*$,
\begin{equation}\label{map coordinates}
    f_{-i}(\lambda_{\bf{R}},\varphi_{\bf{S}}) = \frac{1}{s}\log \circ C_i \circ \mathfrak{L}_s(\lambda_{\bf{R}},\varphi_{\bf{S}}).
\end{equation}
The fiber $F_{s\xi}^{-1}(\lambda_{\bf{R}^-})$ is the symplectic leaf $\mathcal{N}_{s\xi}$. The following lemma will allow us to apply the implicit function theorem at certain points in $\mathcal{N}_{s\xi}$.

\begin{Lem}\label{lem 3.1}
    For all $(\lambda_{\bf{R}},\varphi_{\bf{S}}) \in \mathcal{C}^{\delta}\times \mathbb{T}^m$,
    the derivatives
    \begin{equation}\label{derivatives}
        \begin{split}
            D_{\lambda_{\bm{R}^-}} F_{s\xi} & = I_r + O(e^{2s\delta}); \\
            D_{\lambda_{\bm{R}^+}} F_{s\xi} & =  O(e^{2s\delta}); \\
            D_{\varphi_{\bm{S}}} F_{s\xi} & =  O(e^{2s\delta}). \\
        \end{split}
    \end{equation}
    (Here $I_r$ is the $r\times r$ identity matrix and $O(e^{s\delta})$ denotes a matrix of the appropriate dimensions whose entries are $O(e^{2s\delta})$ as functions of $s$.) 
\end{Lem}

\begin{proof} By the formula for $f_{-i}$, Equations \eqref{master equations}, and the comment directly following Equations \eqref{master equations}, 
\begin{equation*}
    \begin{split}
            e^{2sf_{-i}(\lambda_{\bm{R}},\varphi_{\bm{S}})} & = e^{2s\lambda_{-i}}\left(1 +  \sum_{j,k}c_{j,k}e^{2sL_{j,k}(\lambda_{\bm{R}},\varphi_{\bm{S}})}\right).\\
    \end{split}
\end{equation*}
for $-i= -r,\ldots, -1$, constants $c_{j,k}$, and some linear combinations $L_{j,k}(\lambda_{\bm{R}},\varphi_{\bm{S}})$.  Differentiating these equations gives
\begin{equation*}
    \begin{split}
            \frac{\partial f_{-i}}{\partial \lambda_k} & = e^{2s(\lambda_{-i} - f_{-i}(\lambda_{\bm{R}},\varphi_{\bm{S}})) }\left(\delta_{-i,k} +  \sum_{j,k}\left(\frac{\partial L_{j,k}}{\partial \lambda_k} + \delta_{-i,k} \right)c_{j,k}e^{2sL_{j,k}(\lambda_{\bm{R}},\varphi_{\bm{S}})}\right);\\
            \frac{\partial f_{-i}}{\partial \varphi_k} & =  e^{2s(\lambda_{-i} - f_{-i}(\lambda_{\bm{R}},\varphi_{\bm{S}})) }\sum_{j,k}\frac{\partial L_{j,k}}{\partial \varphi_k}c_{j,k}e^{2sL_{j,k}(\lambda_{\bm{R}},\varphi_{\bm{S}})}.\\
    \end{split}
\end{equation*}
Here $\delta_{-i,k}$ is the Kronecker-delta function. By Proposition \ref{prop 2.4}, for $(\lambda_{\bf{R}},\varphi_{\bf{S}}) \in \mathcal{C}^{\delta}\times \mathbb{T}^m$, 
\begin{equation*}
    \begin{split}
        e^{2s(\lambda_{-i} - f_{-i}(\lambda_{\bm{R}},\varphi_{\bm{S}})) } & = 1 + O(e^{2s\delta}); \\
        e^{2sL_{j,k}(\lambda_{\bm{R}},\varphi_{\bm{S}})} & = O(e^{2s\delta}),
    \end{split}
\end{equation*}
which completes the proof.
\end{proof}


Fix an arbitrary element $p = (\lambda_{\bf{R}^-},\lambda_{\bf{R}^+},\varphi_{\bf{S}}) \in \mathcal{P}_{\xi}$ and consider the subspace 
\[
    \mathcal{S}_p := \mathbb{R}^r \times \{ \lambda_{\bf{R}^+}\} \times \{ \varphi_{\bf{S}}\} \subseteq \mathbb{R}^r \times \mathbb{R}^m \times \mathbb{T}^m.
\]
By an intermediate value theorem argument, we can show that for $s\ll 0$, $\mathcal{N}_{s\xi}$ intersects $\mathcal{S}_p$ near $p$:

\begin{Lem}\label{lem 3.2}
    For all $\xi \in \mathfrak{t}_+^*$ and for all $\delta,\upsilon >0$ sufficiently small, there exists $s_0<0$ such that for all $s\leqslant s_0$ and $p\in \mathcal{P}_{\xi}^{\delta}$, the intersection $\mathcal{S}_p \cap \mathcal{N}_{s\xi} \cap B_{\upsilon}(\mathcal{P}_{\xi})$ is non-empty (see Figure \ref{fig 2}).
\end{Lem}

\begin{figure}[htp]
\centering
	\begin{tikzpicture}[baseline={([yshift=-.5ex]current bounding box.center)}, scale=1.6]
		\fill [blue!30] (1.7,0.5) -- (2.3,.5) -- (2.3,-.5) -- (1.7,-0.5) -- cycle;
		
		\draw (0,0)--(4,2);
		\draw (0,0)--(4,-2);
		\draw (4.5,2.1)  node[align=left] {$\mathcal{C}\times \mathbb{T}^m$};
		
		\draw (.5,0)--(4,1.75);
		\draw (.5,0)--(4,-1.75);
		\draw (4.62,1.7)  node[align=left] {$\mathcal{C}^{\delta/2}\times \mathbb{T}^m$};
		
		\draw (2,-1)--(2,1);
		\draw [line width=.5mm](2,-.5)--(2,.5);
		
		\draw (1.7,-1)--(1.7,1);
		\draw (2.3,-1)--(2.3,1);
		\draw (2.3,1) arc(0:180:.3);
		\draw (2.3,-1) arc(0:-180:.3);
		\draw (2,1.6) node [align=right] {$B_{\upsilon}(\mathcal{P}_{\xi})$};
		
		\draw (-1,.27)--(4,.27);
		\draw (-1.2,.27)  node {$\mathcal{S}_{p}$};
		\draw (2,.27) node {}[circle,fill,inner sep=.04pt];
		\filldraw[black] (2,.27) circle (1pt);
		\draw (2.1,0) node [align=left,above] {$p$};
		
		
		\draw [->] (2.55,-.75)--(2.05,-.6);
		\draw (2.7,-.8) node [align=left] {$\mathcal{P}_{\xi}$};
		
		\draw[domain=-.995:.995,smooth,variable=\t,line width=.2mm]plot ({ln(2*cosh(2) - 2*cosh(2*\t))/2+1},\t);
		\draw (-.7,1)  node[align=left] {$\mathcal{N}_{s\xi}$};
		\filldraw[black] (1.83,.27) circle (1pt);

	\end{tikzpicture}

\caption{The intersection  described in Lemma \ref{lem 3.2}. The intersection of $\mathcal{N}_{s\xi}$ with the shaded region is locally the graph of a function defined on $\mathcal{P}_{\xi}^{\delta}$ (Proposition \ref{prop 3.3}). In the figure, $\mathcal{P}_{\xi}^{\delta}$ is the thick part of $\mathcal{P}_\xi$.}
\label{fig 2}
\end{figure}
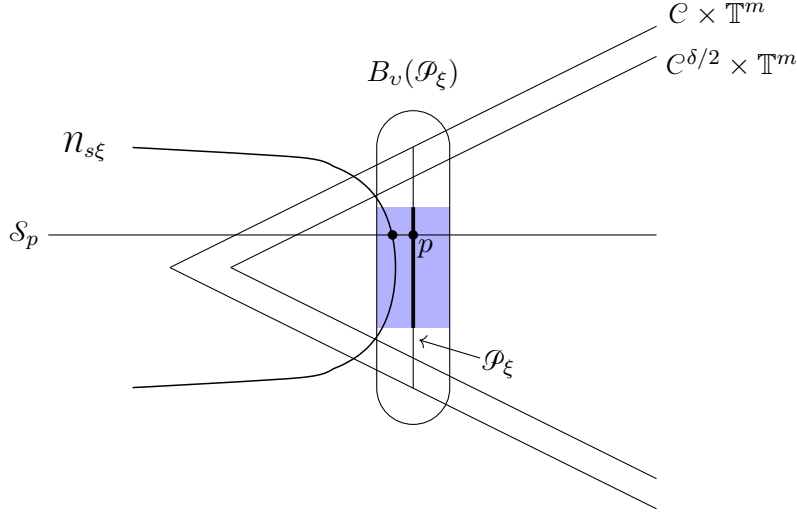

\begin{proof} Consider the equivalent problem of showing there is a $s_0$ such that for all $s\leqslant s_0$ and $p \in \mathcal{P}_{\xi}^{\delta}$, the submanifold $\mathfrak{L}_s(\mathcal{S}_p  \cap B_{\upsilon}(\mathcal{P}_{\xi}))$ intersects the dressing orbit $\mathcal{D}_{s\xi}$. Since dressing orbits are joint level sets of the Casimir functions $C_i$, showing this intersection is non-empty is equivalent to showing that $\lambda_{\bf{R}^-}$ is contained in the image of $\mathcal{S}_p\cap B_{\upsilon}(\mathcal{P}_{\xi})$ under the map $F_{s\xi}$ defined in Equations \eqref{map} and \eqref{map coordinates}.

Fix $\delta>0$ (small enough that $\mathcal{P}_{\xi}^{\delta}$ is nonempty). By Corollary \ref{limit corollary}, for $\varepsilon >0$ sufficiently small,
\[
    \lim_{s\to -\infty}f_{-i}( \lambda_{-r},\ldots, \lambda_{-i}\pm \varepsilon, \ldots, \lambda_{-1},\lambda_{\bf{R}^+},\varphi_{\bf{S}}) =  \lambda_{-i} \pm \varepsilon.
\]
Thus, for all $p \in \overline{\mathcal{P}}_{\xi}^{\delta}$, there is a $s_p$ such that for $s\leqslant s_p$, the map $F_{s\xi}$ satisfies the assumptions of the Poincar\'e-Miranda Theorem on the box
\[
    [\lambda_{-r}-\varepsilon,\lambda_{-r}+\varepsilon]\times \cdots \times [\lambda_{-1}-\varepsilon,\lambda_{-1}+\varepsilon] \times \{\lambda_{\bm{R}^+}\}\times\{\varphi_{\bm{S}}\} \subset \mathcal{S}_p.
\]
Take $\varepsilon>0$ sufficiently small so that the box is contained in $\mathcal{S}_p\cap B_{\upsilon}(\mathcal{P}_{\xi})$ and, without loss of generality (making $\upsilon$ smaller if necessary), assume that  $\mathcal{S}_p\cap B_{\upsilon}(\mathcal{P}_{\xi}) \subset \mathcal{C}^{\delta/2}\times \mathbb{T}^m$ for all $p \in \overline{\mathcal{P}}_{\xi}^{\delta}$.  It follows by the Poincar\'e-Miranda theorem that $\lambda_{\bf{R}^-}$ is contained in the image of the box under the map $F_{s\xi}$ for $s \leqslant s_p$.

By transversality of the intersection of $\mathcal{S}_p$ and $\mathcal{N}_{s\xi}$ at points in $\mathcal{C}^{\delta/2}\times \mathbb{T}^m$, for $s$ less than some $s'$ (Lemma \ref{lem 3.1}), each $p\in \overline{\mathcal{P}}_{\xi}^{\delta}$ has a neighborhood $U_p$ such that for $p' \in U_p$ and $s \leqslant s_p$, the intersection $\mathcal{S}_{p'} \cap \mathcal{N}_{s\xi} \cap B_{\upsilon}(\mathcal{P}_{\xi})$ is non-empty. Passing to a finite subcover $U_{p_k}$, $k = 1,\ldots, n$ and letting $s_0 = \min\{s',s_{p_k}\}$ completes the proof.
\end{proof}

Define 
\[
    \mathcal{U}_{\xi,\delta} := \bigcup_{p\in \mathcal{P}_{\xi}^{\delta}} \mathcal{S}_p.
\]
From this point forward, take $\upsilon>0$ sufficiently small so that $\mathcal{U}_{\xi,\delta} \cap B_{\upsilon}(\mathcal{P}_{\xi}) \subset \mathcal{C}^{\delta/2}\times \mathbb{T}^m$. The region $\mathcal{U}_{\xi,\delta} \cap B_{\upsilon}(\mathcal{P}_{\xi})$ is shaded blue in Figure \ref{fig 2}.

\begin{Pro}\label{prop 3.3}
For all $\delta>0$ and $s\leqslant s_0$ as in Lemma \ref{lem 3.2}, the intersection $\mathcal{N}_{s\xi}\cap \mathcal{U}_{\xi,\delta} \cap B_{\upsilon}(\mathcal{P}_{\xi})$ is locally the graph of a function 
\[
    g_s \colon \mathcal{P}_{\xi}^{\delta} \to \mathbb{R}^r.
\]
\end{Pro}

\begin{proof} Combine Lemmas \ref{lem 3.1}, \ref{lem 3.2}, and the implicit function theorem.  
\end{proof}

\subsection{Comparing symplectic volumes on the leaves of \texorpdfstring{$\pi_s$}{pis}}\label{ss 3.2}

In this subsection, we compare the symplectic volumes of $(\mathcal{P}_{\xi},\nu_{\xi})$ and $(\mathcal{N}_{s\xi},\eta_{s\xi})$. By Proposition \ref{prop 3.3}, the intersection of $\mathcal{N}_{s\xi}$ with $ \mathcal{U}_{\xi,\delta} \cap B_{\upsilon}(\mathcal{P}_{\xi})$ is locally the graph of a function $g_s$. i.e. locally  there is a diffeomorphism 
\[
    G_s\colon \mathcal{P}_{\xi}^{\delta} \to \mathcal{N}_{s\xi},\, \left(\lambda_{\bf{R}},\varphi_{\bf{S}}) \mapsto (g_s(\lambda_{\bf{R}^+},\varphi_{\bf{S}}),\lambda_{\bf{R}^+},\varphi_{\bf{S}}\right)
\]

\begin{Lem}\label{lem 3.4} For $s\leqslant s_0$ as in Lemma \ref{lem 3.2}, at points in $\mathcal{N}_{s\xi}\cap  \mathcal{U}_{\xi,\delta} \cap B_{\upsilon}(\mathcal{P}_{\xi}) \subset \mathcal{C}^{\delta/2}\times \mathbb{T}^m$, 
	\[
		(G_s)_*\nu_{\xi} = \eta_{s\xi}+ O(e^{s\delta })	\]
(here $O(e^{s\delta})$ denotes a 2-form whose coefficients in  coordinates $(\lambda_{\bf{R}},\varphi_{\bf{S}})$ are $O(e^{s\delta})$ as functions of $s$).
\end{Lem}

\begin{proof} Fix $p = (\lambda_{\bf{R}},\varphi_{\bf{S}}) \in \mathcal{P}_{\xi}^{\delta}$. By the implicit function theorem, for all $(X,Y)  \in T_p\mathcal{P}_{\xi}^{\delta} = \mathbb{R}^m \times \mathbb{R}^m$,
\[
        D_pG_s(X,Y) = \left(-(D_{\lambda_{\bf{R}^-}}F_{s\xi} )^{-1}(D_{\lambda_{\bf{R}^+}}F_{s\xi} X +D_{\varphi_{\bf{S}}} F_{s\xi}Y),X,Y\right)
\]
The constant bivector $\pi_{PT}$ has the form 
\[
    \pi_{PT} = \sum_{k} X_k\wedge Y_k
\]
for some $X_k, Y_k \in T_p\mathcal{P}_{\xi}^{\delta}$.
By Lemma \ref{lem 3.1} and the formula for $D_pG_s$ above, we find $(G_s)_*\pi_{PT} = \pi_{PT} + O(e^{s \delta}),$ where $O(e^{s \delta})$ denotes a bivector whose coefficients in coordinates $(\lambda_{\bf{R}},\varphi_{\bf{S}})$ are $O(e^{s \delta})$ as functions of $s$. The 2-form  \[
(G_s)_*\nu_{\xi} = ((G_s)_*\pi_{PT})^{-1} = \pi_{PT}^{-1} + O(e^{s\delta}).
\]
On the other hand, by the proof of \cite[Theorem 6.18]{ABHL}, at $G_s(p) \in  \mathcal{C}^{\delta/2}\times \mathbb{T}^m$,
\[
    \eta_{s\xi} = (\pi_s)^{-1} = \left(\pi_{PT}+ O(e^{s\delta})\right)^{-1} =  \pi_{PT}^{-1} + O(e^{s\delta}). \tag*{\qedhere}
\]
\end{proof}

Finally, we show that for $s\ll 0$, the symplectic volume of $\mathcal{N}_{s\xi}$ is concentrated on the piece that lies in $\mathcal{C}^{\delta/2}\times \mathbb{T}^m$. 

\begin{Pro}\label{prop 3.5}
For $\xi$, $\delta$ ,$\upsilon$, and $s\leqslant s_0$ as in Lemma \ref{lem 3.2}, the symplectic volume of  $\mathcal{N}_{s\xi}\cap \mathcal{U}_{\xi,\delta} \cap B_{\upsilon}(\mathcal{P}_{\xi}) \subset \mathcal{C}^{\delta/2}\times \mathbb{T}^m$ satisfies the inequalities
\[
    \vol(\mathcal{N}_{s\xi} , \eta_{s\xi}) \geqslant \vol(\mathcal{N}_{s\xi}\cap \mathcal{U}_{\xi,\delta} \cap B_{\upsilon}(\mathcal{P}_{\xi}) , \eta_{s\xi}) \geqslant \vol(\mathcal{N}_{s\xi} , \eta_{s\xi}) - \vol(\mathcal{P}_{\xi}\setminus\mathcal{P}_{\xi}^{\delta}  , \nu_{\xi}) + O(e^{\delta s}).
\] 
\end{Pro}
Note that $ \vol(\mathcal{P}_{\xi}\setminus\mathcal{P}_{\xi}^{\delta}  , \nu_{\xi})\to 0$ as $\delta \to 0$. 

\begin{Rmk} In the proof of Theorem \ref{maintheorem}, we choose $\delta, \upsilon >0$ sufficiently small and let  $D_\varepsilon$ be the closure of $ \mathcal{U}_{\xi,\delta} \cap B_{\upsilon}(\mathcal{P}_{\xi}) \subseteq \mathcal{C}^{\delta/2}\times \mathbb{T}^m$.
\end{Rmk}

\begin{proof}
The first inequality follows since volume is monotonic. By Proposition \ref{prop 3.3} and Lemma \ref{lem 3.4}, $\mathcal{N}_{s\xi}\cap\mathcal{U}_{\xi,\delta} \cap B_{\upsilon}(\mathcal{P}_{\xi})$ is locally the image of a diffeomorphism $G_s$ with domain in $\mathcal{P}_{\xi}^{\delta}$ and  $(G_s)_*\nu_{\xi} = \eta_{s\xi}+ O(e^{s\delta })$, so 
  \[
	\vol(\mathcal{N}_{s\xi}\cap\mathcal{U}_{\xi,\delta} \cap B_{\upsilon}(\mathcal{P}_{\xi}), \eta_{s\xi}) \geqslant  \vol(\mathcal{P}_{\xi}^{\delta} , \nu_{\xi})  + O(e^{s\delta}).
\]
By definition of $\mathcal{P}_{\xi}^{\delta} = \mathcal{P}_{\xi} \cap (\mathcal{C}^{\delta}\times \mathbb{T}^m)$,
\[
  \vol(\mathcal{P}_{\xi}^{\delta} , \nu_{\xi}) = \vol(\mathcal{P}_{\xi} , \nu_{\xi})  - \vol(\mathcal{P}_{\xi}\setminus\mathcal{P}_{\xi}^{\delta}  , \nu_{\xi}). 
\]
Finally, by Theorem \ref{theoremcomparevolume},
\[
	\vol(\mathcal{P}_{\xi} , \nu_{\xi})  - \vol(\mathcal{P}_{\xi}\setminus\mathcal{P}_{\xi}^{\delta}  , \nu_{\xi}) + O(e^{s\delta/2})  = \vol(\mathcal{N}_{s\xi} , \eta_{\xi}) - \vol(\mathcal{P}_{\xi}\setminus\mathcal{P}_{\xi}^{\delta}  , \nu_{\xi})  + O(e^{s\delta}).  \tag*{\qedhere}
\]
\end{proof}

\section{Preimages of points in \texorpdfstring{$PT(K^*)$}{PT(K*)}}\label{section4}

The goal of this section is to show that for a fixed value of $\xi \in \mathfrak{t}_+^*$ and $s\ll0$, if $\mathfrak{E}_s(\Ad_k^* \xi)\in \mathfrak{L}_s(PT(K^*))$, then $\Ad_k^* \xi$ must be close to $ \xi$ in the coadjoint orbit $\mathcal{O}_\xi$. 

Fix a faithful irreducible representation $(\rho, V)$ of $G$. Let $n=\dim (V)$, and fix a Hermitian inner product on $V$ which is preserved by $\rho(K)$. For the representation $V$, fix a unitary weight basis $v_1,\dots, v_n$.  Consider the wedge product $(\rho^l,\wedge^l V)$ of the representation $(\rho, V)$. Note that $\wedge^l V$ has basis
  \[
    \{v_{\bf{I}}:=v_{i_1}\wedge \cdots \wedge v_{i_l} \mid {\bf I} =(i_1,\dots,i_l) \text{~and~} i_1<\cdots<i_l\}.
  \]
  We can reorder the unitary weight basis $\{v_i\}$ so that, for all $l\in [n]$, the vector $v_{[l]}=v_1\wedge \cdots \wedge v_l$ is a minimal weight vector of $\wedge^l V$. For ${\bf I},{\bf J}\subset [n]$ with $|{\bf I}|=|{\bf J}|=l$ denote by $\Delta_{{\bf I},{\bf J}}$ the $l\times l$ minor of elements of $\GL(V)$ in the basis $v_i$, with rows ${\bf I}$ and columns ${\bf J}$.
Define the map 
\[
    \pr_{\mathfrak{t}^*} \colon PT(K^*)\to \mathfrak{t}^*;\qquad x\in \mathcal{P}_\xi\mapsto \xi.
\]

\begin{Lem} \label{mainlemma}
  Let $l\in [n]$, and let ${\bf J}\subset [n]$ with $|{\bf J}|=l$ and $[l]\ne {\bf J}$. For all $\delta >0$ and $s< 0$, define 
\[
    U_s = \{ k \in K \mid \mathfrak{E}_s(\Ad^*_k\xi) = \mathfrak{L}_s(p) \mbox{ for some } p \in \mathcal{C}^\delta\times \mathbb{T}^m, \, \xi \in \pr_{\mathfrak{t}^*}(\mathcal{C}^\delta\times \mathbb{T}^m) \}.
\] 
Then there exists $a>0$ such that for all $k \in U_s$, 
\[
    |\Delta_{[l],{\bf J}}(\rho(k))| \leqslant a e^{s\delta},
\]
in the unitary weight basis $\{v_i\}$.
\end{Lem}

\begin{proof}
 Let $\wt(v_{[l]})=w_0 \zeta$, where $\zeta\in P_+$ is a dominant weight, and consider the irreducible subrepresentation $G\cdot v_{[l]}$ of $\wedge^l V$ which is generated by $v_{[l]}$. Then in this subrepresentation, $v_{[l]}$ will be of lowest weight. Let ${\bf L}$ denote the index of the highest weight vector of this subrepresentation. It follows that $\wt(v_{\bf L})=\zeta$. Write the matrix entries of $\rho^l(g)$ in the basis $\{v_{\bf{I}}\}$ as $\rho^l_{{\bf{I}},{\bf{J}}}(g)$.
  Note that $\rho^l_{{\bf I},{\bf J}}(g)=\Delta_{{\bf I},{\bf J}}(\rho(g))$. 
  Because $v_{[l]}$ is of lowest weight in the subrepresentation $G\cdot v_{[l]}$, we have
 \begin{equation}\label{rep}
     \rho^l(g)v_{[l]}=\sum_{\substack{ w_0 \zeta<\wt(v_{\bf J}) \\ \text{or }{\bf J} = {[l]}} }  \rho^l_{{\bf J},{[l]}}(g)v_{{\bf J}},
 \end{equation}
 where the sum on the right hand side is over weight vectors $v_{\bf J}$ such that $w_0 \zeta-\wt(v_{\bf J})$ is a negative weight or ${\bf J} = [l]$. In other words, $\rho^l_{{\bf{J}},[l]}(g)=0$ unless $w_0 \zeta<\wt(v_{\bf{J}})$ or ${\bf J}=[l]$. 

Using the definition of the dressing action and the fact that the map $\mathfrak{E}_s$ is $K$-equivariant, we have
\begin{equation}\label{master}
    k\cdot\left(\mathfrak{E}_s(\xi)\right)^2\cdot k^*= \mathfrak{E}_s(\Ad_k^* \xi)\cdot \mathfrak{E}_s(\Ad_k^*\xi)^*.
\end{equation}
Rewrite \eqref{master} as
  \begin{equation}\label{master2}
    k\cdot d_s^2\cdot k^*=b_s\cdot b_s^*
  \end{equation}
  where $d_s=\exp\left( s \sqrt{-1} \psi(\xi)\right)$ and $b_s=\mathfrak{L}_s(p)$. 
  
Let us apply the representation $\rho^l$ to both sides of \eqref{master2}, and consider the $([l],[l])$-entry of these matrices. Using the fact that $\{v_{\bf I}\}$ is a unitary basis for $\wedge^l V$,  matrix multiplication and \eqref{rep} gives us: 
 \begin{equation}\label{eq:1}
     \sum_{\substack{ w_0 \zeta<\wt(v_{\bf J}) \\ \text{or }{\bf J} = {[l]}} }  \big|\rho^l_{{\bf J},{[l]}}(k^*)\big|^2\cdot \big|\rho^l_{{\bf J},{\bf J}}(d_s)\big|^2= \sum_{\substack{ w_0 \zeta<\wt(v_{\bf J})\\ \text{or }{\bf J} = {[l]} } }   \big|\rho^l_{{\bf J},{[l]}}(b_s^*)\big|^2.
 \end{equation}
  Since $\rho^l(k)\cdot \rho^l(k^*)=\rho^l(kk^*)=1$, we have
  \begin{equation}\label{eq:2}
    \sum_{\substack{  w_0 \zeta< \wt(v_{\bf J}) \\ \text{or }{\bf J} = {[l]}}  } \big|\rho^l_{{\bf J},{[l]}}(k^*)\big|^2=1.
  \end{equation}
  Rewrite \eqref{eq:2} as 
  \[
    \big|\rho^l_{[l],[l]}(k^*)\big|^2=1- \sum_{ w_0 \zeta<\wt(v_{\bf J}) }  \big|\rho^l_{{\bf J},{[l]}}(k^*)\big|^2
  \] 
  and plug it into \eqref{eq:1}. After rearranging, we get
  \begin{equation}
     \big|\rho^l_{{[l]},{[l]}}(d_s)\big|^2=\sum_{\substack{  w_0 \zeta< \wt(v_{\bf J}) \\ \text{or }{\bf J} = {[l]}}  }  \big|\rho^l_{{\bf J},{[l]}}(b_s^*)\big|^2 \label{rhosquared} 
     +\sum_{w_0 \zeta< \wt(v_{\bf J})}  \big|\rho^l_{{\bf J},{[l]}}(k^*)\big|^2\cdot \left(\big|\rho^l_{{[l]},{[l]}}(d_s)\big|^2 - \big|\rho^l_{{\bf J},{\bf J}}(d_s)\big|^2\right).
  \end{equation}
 Since $w_0 \zeta<\wt(v_{\bf L})$ and the terms $\big|\rho^l_{{[l]},{[l]}}(d_s)\big|^2 - \big|\rho^l_{{\bf J},{\bf J}}(d_s)\big|^2$ are positive, by discarding terms on the right hand side of \eqref{rhosquared}, one has for any ${\bf J}$ with $w_0 \zeta<\wt(v_{\bf J})$,
  \[
    \big|\rho^l_{{[l]},{[l]}}(d_s)\big|^2>\big|\rho^l_{{\bf L},[l] }(b_s^*)\big|^2+ \big|\rho^l_{{\bf J},{[l]}}(k^*)\big|^2\cdot \left(\big|\rho^l_{{[l]},{[l]}}(d_s)\big|^2 - \big|\rho^l_{{\bf J},{\bf J}}(d_s)\big|^2\right).
  \] 
  Hence
  \begin{equation}\label{eq:3}
    \big|\rho^l_{{\bf J},{[l]}}(k^*)\big|^2<\frac{ \big|\rho^l_{{[l]},{[l]}}(d_s)\big|^2-\big|\rho^l_{{\bf L},{[l]}}(b_s^*)\big|^2}{\big|\rho^l_{{[l]},{[l]}}(d_s)\big|^2 - \big|\rho^l_{{\bf J},{\bf J}}(d_s)\big|^2}=\frac{ 1-\big|\rho^l_{{[l]},{\bf L}}(b_s)\big|^2/\big|\rho^l_{{[l]},{[l]}}(d_s)\big|^2}{1 - \big|\rho^l_{{\bf J},{\bf J}}(d_s)\big|^2/\big|\rho^l_{{[l]},{[l]}}(d_s)\big|^2}.
  \end{equation}
  
   From Proposition \ref{prop 2.4}, because $p\in \mathcal{C}^\delta\times \mathbb{T}^m$, we have
      \[
      C_i(b_s)^2= |\Delta_{w_0\omega_i,\omega_i} (b_s)|^2\left(1+O(e^{2s\delta})\right).
      \]
      On the other hand, from \eqref{master}, for $s< 0$,
      \begin{align*}
      C_i(b_s)^2 & = \Tr(\rho^{\omega_i}(d_s^2))
      = \sum_{j} c_j e^{2s (\gamma_j,\sqrt{-1}\xi)}  
      = e^{2s(w_0\omega_i,\sqrt{-1} \xi)}\left(1+O(e^{2s\delta})\right).
      \end{align*}
      Here, the weights $\gamma_j$ are those which appear in the representation $\rho^{\omega_i}$, and $c_j=1$ when $\gamma_j$ is the extremal weight $w_0\omega_i$. The last equality holds because, by assumption, $\xi\in \pr_{\mathfrak{t}^*}(\mathcal{C}^\delta\times \mathbb{T}^m)$, which in turn guarantees that $(\alpha_i,\sqrt{-1} \xi)>\delta$ for all $i\in I$.
      
      Combining the previous two equations, since $e^{s(w_0\omega_i,\sqrt{-1}\xi)}=\Delta_{w_0\omega_i,w_0\omega_i}(d_s)$, we have
  \[
    \left | \left| \frac{\Delta_{w_0\omega_i,\omega_i}(b_s)}{\Delta_{w_0\omega_i,w_0\omega_i}(d_s)}\right|^2-1 \right|= O(e^{2s\delta}), \quad \text{~for all~} i \in I.
  \]
    For $\zeta\in P_+$, by using \eqref{deltadef} we get
  \begin{equation}\label{eq:0}
     \left | \left| \frac{\Delta_{ w_0 \zeta,\zeta}(b_s)}{\Delta_{w_0\zeta,w_0\zeta}(d_s)}\right|^2-1 \right|=O(e^{2s\delta}),
  \end{equation}
  for $s\ll 0$.
    By the discussion at the end of Section \ref{section2}, we know
  \[
    \rho^l_{[l],[l]}=c \Delta_{ w_0 \zeta,w_0 \zeta} \quad \text{and} \quad \rho^l_{[l],{\bf L}}=c \Delta_{w_0\zeta,\zeta}
  \]
  for some $c\in \mathbb{C}^\times$. By \eqref{eq:0} and \eqref{eq:3}, we find $|\Delta_{[l],{\bf J}}(\rho(k))|=|\Delta_{{\bf J},{[l]}}(\rho(k^*))|=O(e^{s\delta})$.
\end{proof}

\begin{Lem}\label{lem 4.2}
Let $g \colon (-\infty,0) \to \UU(n)$ be an element of $\UU(n)$ depending on a parameter $s$. Assume there exists $\delta>0$ such that
\[|\Delta_{[l],{\bf{J}}}(g(s))|=O(e^{s\delta})\qquad  \text{for all }l\in [n]  \text{ and all }{\bf J}\subset [n]\text{ with } |{\bf J}|=l\text{ and }[l]\ne {\bf J}.\] Then, the matrix entries satisfy $|g_{i,j}(s)|=O(e^{s\delta})$ for all $i\neq j$.
\end{Lem}

\begin{proof}
We proceed by induction on $i$. When $i=1$, we have $|g_{1,j}| = O(e^{s\delta})$ for $j\ne 1$. Assume the statement is known for $1,\dots, i-1$. By induction hypothesis and the fact that $g$ is unitary, we have $1-|g_{j,j}|=O(e^{s\delta})$ for $j<i$. By taking inner product of the $i^{th}$ row with the previous rows and again using the fact that $g$ is unitary, we have $|g_{i,j}|=O(e^{s\delta})$ for $j<i$. For $j>i$, consider the minor $\Delta_{[i],{\bf J}}(g)$, where ${\bf J}=\{1,\dots, i-1, j \}$. By assumption, $|\Delta_{[i],{\bf J}}(g)|=O(e^{s\delta})$. Expanding this minor along the $j^{th}$ column and applying the induction hypothesis, we have that $|g_{i,j}|=O(e^{s\delta})$. 
\end{proof}

Recall that $\mathcal{N}_{s\xi}$ is the preimage $(\mathfrak{E}_s\n\circ \mathfrak{L}_s)\n(\mathcal{O}_{\xi})$.

\begin{Pro}
\label{lastcor}

For all $\xi\in \mathfrak{t}^*_+$, if $U\subset \mathcal{O}_\xi$ is an open subset with $\xi\in U$, then for all sufficiently small $\delta>0$, there exists $s_0\in \mathbb{R}$ so that, for all $s\leqslant s_0$, 
\[
\mathfrak{E}_s\n\circ \mathfrak{L}_s\left(\mathcal{N}_{s\xi}\cap(\mathcal{C}^\delta\times \mathbb{T}^m)\right)\subseteq U.
\]
\end{Pro}

\begin{proof} Fix $\xi \in \mathfrak{t}_+^*$, $U\subseteq \mathcal{O}_\xi$ open with $\xi \in U$, and $\delta >0$ sufficiently small so that $\xi \in \pr_{\mathfrak{t}^*}(\mathcal{C}^\delta\times \mathbb{T}^m)$.  Observe that for all $s<0$, 
\[
    U_s' = \{ k \in K \mid \mathfrak{E}_s(\Ad_k^* \xi) \in \mathfrak{L}_s(\mathcal{N}_{s\xi} \cap (\mathcal{C}^{\delta} \times \mathbb{T}^m)) \} \subseteq U_s.
\]
By Lemma \ref{mainlemma}, there exists $a >0$ such that for all $k \in U_s'$, 
\[
    |\Delta_{[l],{\bf J}}(\rho(k))| \leqslant a e^{s\delta}.
\]
By Lemma \ref{lem 4.2} and since $\rho$ faithful, there exists $s_0<0$ such that for all $s\leqslant s_0$, 
\[
    \mathfrak{E}_s\n\circ \mathfrak{L}_s\left(\mathcal{N}_{s\xi}\cap(\mathcal{C}^\delta\times \mathbb{T}^m)\right)\subseteq U.\tag*{\qedhere}
\]
\end{proof}

\Addresses

\end{document}